\documentclass[12pt]{article}
\usepackage{amsmath,amssymb,amsthm}
\usepackage{color,url,array,multirow,graphicx}
\usepackage{tikz}

\newtheorem{theorem}{Theorem}
\newtheorem{corollary}[theorem]{Corollary}
\newtheorem{lemma}[theorem]{Lemma}

\theoremstyle{definition}
\newtheorem{definition}{Definition}

\newcommand{\bx}{\mathbf{x}}
\newcommand{\by}{\mathbf{y}}
\newcommand{\bz}{\mathbf{z}}

\begin{document}
\title{Inglenook Shunting Puzzles}
\author{Simon R. Blackburn\\
Department of Mathematics\\
Royal Holloway University of London\\
Egham, Surrey TW20 0EX, United Kingdom\\
{\tt s.blackburn@rhul.ac.uk}.}
\maketitle

\vspace{-1cm}

\begin{abstract}
An inglenook puzzle is a classic shunting (switching) puzzle often found on model railway layouts. A collection of wagons sits in a fan of sidings with a limited length headshunt (lead track). The aim of the puzzle is to rearrange the wagons into a desired order (often a randomly chosen order). This article answers the question: When can you be sure this can always be done? The problem of finding a solution in a minimum number of moves is also addressed.
\end{abstract}

\section{Introduction}

This paper provides an analysis of when an inglenook puzzle can be solved, and how many moves are needed in the worst case. Most of the paper assumes that the reader has a background in discrete mathematics or computer science, but the first part of this introduction provides a summary of the results of the paper for readers who do not necessarily have this background. The remaining parts of this introduction describes some of the previous academic work on related problems, and describes the structure of the rest of the paper.

\subsection{A non-technical summary}
\label{sub:nontechnical}

Puzzles involving the movement of locomotives, wagons and carriages have a long history, with well-known examples such as Sam Loyd's \emph{Primitive Railroading} puzzle and \emph{The Switch Problem}~\cite{Loyd} and Dudeney's \emph{The Mudville Railway Muddle}~\cite{Dudeney} dating back well over 100 years. Teun Spaans' Just Puzzles website~\cite{justpuzzles} gives an excellent description and brief discussion of these and similar puzzles; Hordern~\cite{Hordern} describes these puzzles in the wider context of sliding piece puzzles. 

Shunting puzzles (where the aim is to shuffle wagons into particular locations) are, unsurprisingly, popular with railway modellers who design them into layouts for operational interest. The two most famous are \emph{Timesaver} (created by John Allen and published in 1972) and \emph{The Inglenook Sidings} (invented by Alan Wright for the Manchester Model Railway Exhibition in 1979). See Adrian Wymann's excellent shunting puzzles website~\cite{wymann} for more information, including the history, construction and operation of such puzzles; also see the inspiring website and books by Carl Arendt~\cite{arendt,Arendt10,Arendt03,Arendt06} for the related concept of a micro layout. 

This paper studies the Inglenook Sidings and some generalisations. A classic Inglenook Sidings puzzle is depicted in Figure~\ref{fig:layout}. We have three \emph{sidings} fanning out from a single track, the \emph{headshunt}. There are eight wagons, all of the same length, distributed in the sidings. The first two sidings are each long enough for three wagons, and the final siding is long enough for five. The headshunt is long enough for three wagons, plus a shunting engine. Five of the wagons are selected at random (by drawing counters one at a time from a bag, say). The aim is to shunt the wagons so that the five selected wagons lie in the long siding, in the correct order, and the remaining three wagons lie in one of the two remaining sidings. A variation also specifies the order of the three wagons in the shorter siding. We assume that wagons are always left in sidings when not being shunted: they are never parked across points (switches).

\begin{figure}
\[
\begin{array}{c}
\includegraphics[width=125mm]{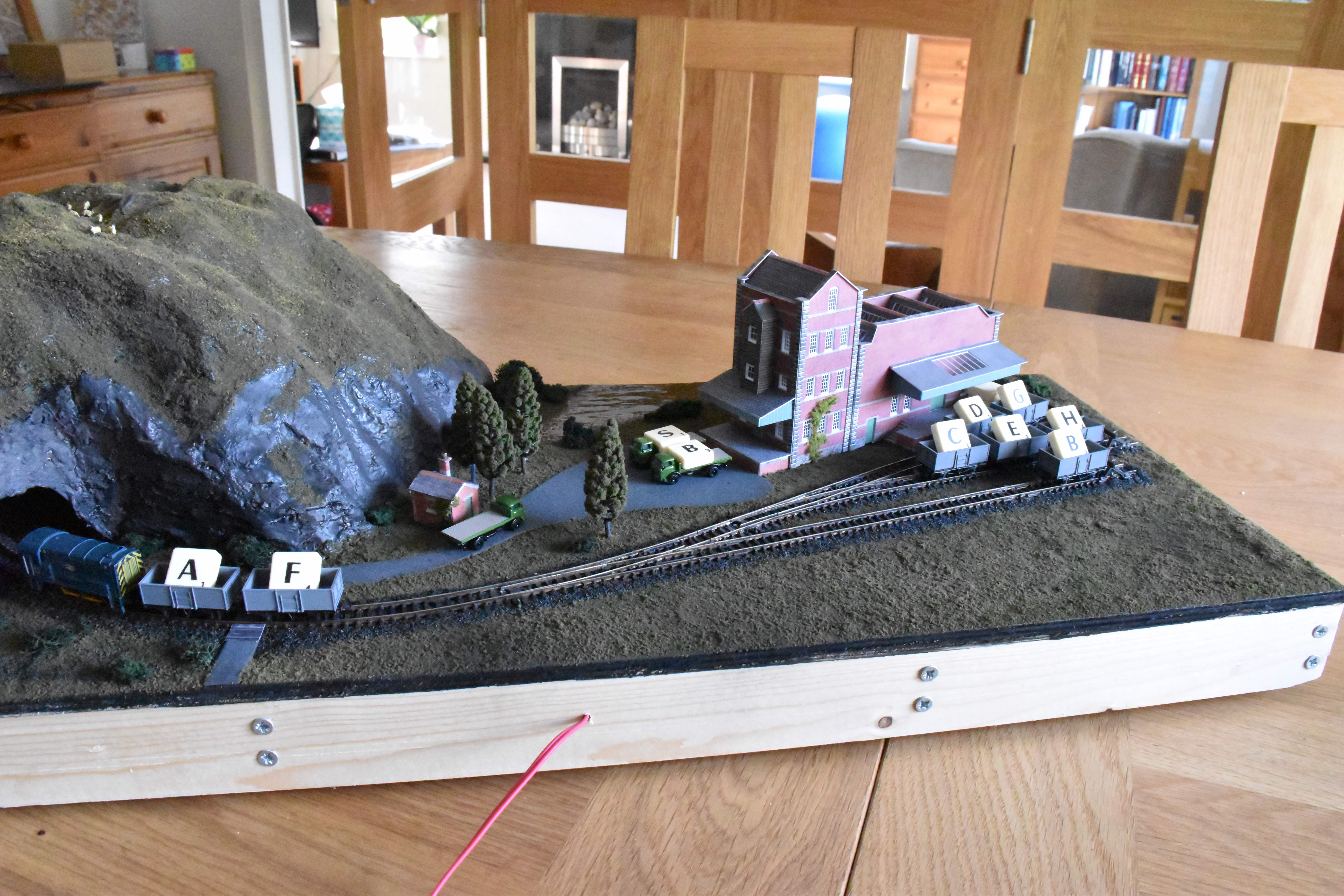}\\
\includegraphics[width=125mm]{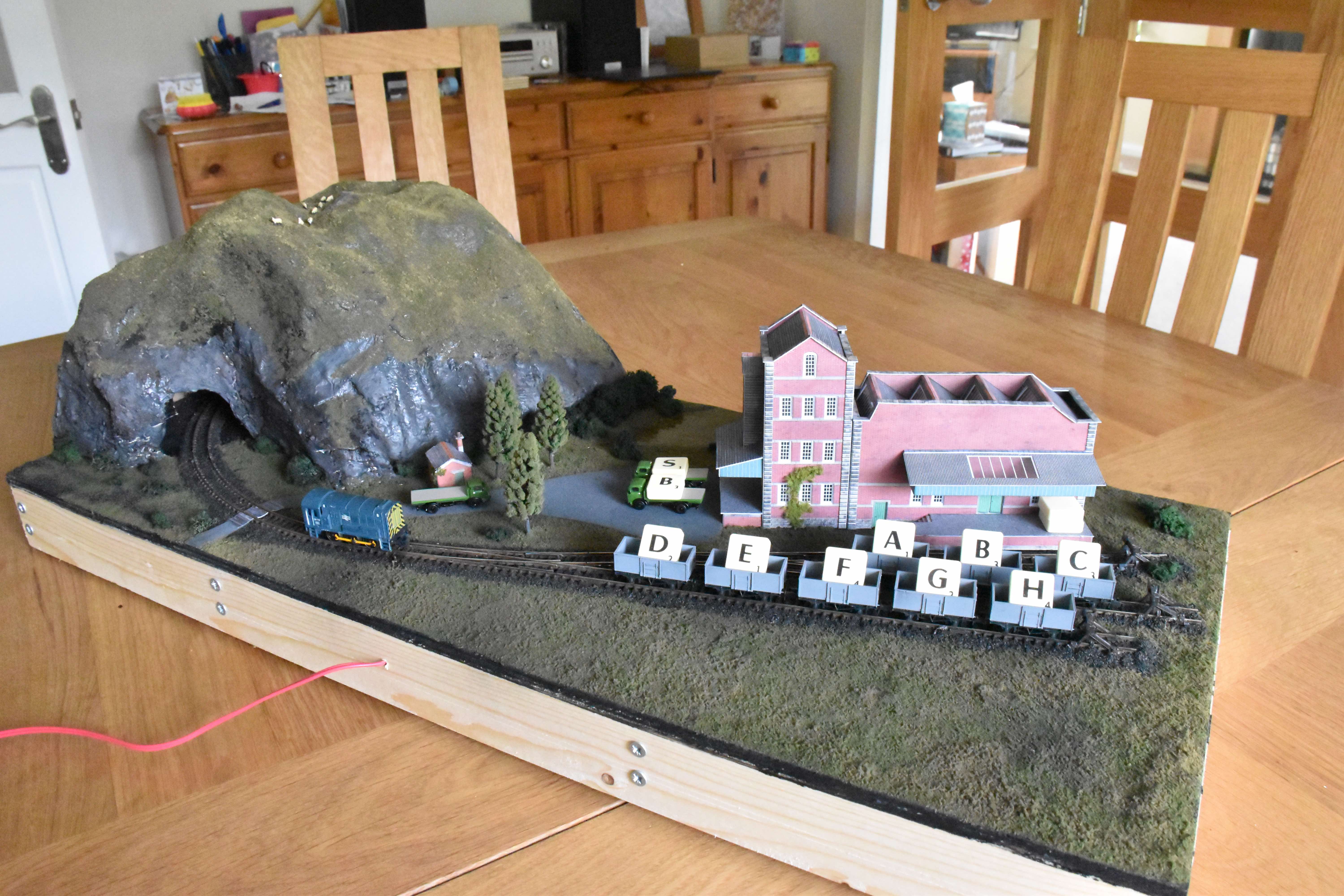}
\end{array}
\]
\caption{An Inglenook Sidings puzzle. Pictured with a long way to go, and then completed.}
\label{fig:layout}
\end{figure}

Experience tells us that the classic Inglenook Sidings puzzle can always be solved, whatever the initial distribution of wagons. (Not all puzzles have this nice property: try solving a `puzzle tray' sliding block puzzle after swapping two adjacent tiles!) In fact, the classic puzzle is small enough so that a computer can be programmed to solve the puzzle in all situations, and so we can be sure that the puzzle can always be solved.

So what happens if the sidings on my layout do not have the lengths of the classic puzzle? For example:
\begin{itemize}
\item[(a)] If I have sidings that can contain $4$, $5$ and $6$ wagons and a headshunt that contains at most 4 wagons, what is the largest number of wagons that works?
\item[(b)] Suppose I have 10 wagons and a headshunt that can contain $3$ wagons. What are the shortest siding lengths I can choose?
\end{itemize}
This paper shows how to answer these kinds of questions. (Answers to the questions above: (a) 12 wagons. (b) You must have a siding of length at least $5$. Sidings of length $4$, $4$ and $5$ work. Sidings of length $3$, $5$ and $5$ also work.) The remainder of this introduction gives an intuition as to what is going on. The sections that follow give a formal proof of the results (using terminology from mathematics/computer science).

More generally we have $w$ wagons, a headshunt that can contain at most~$h$ wagons, and sidings that can contain at most $m_1$, $m_2$ and $m_3$ wagons respectively. (The classical Inglenook Sidings puzzle will have $w=8$, $h=m_1=m_2=3$, and $m_3=5$.) Let $M$ be the maximum value of $h-1$, $m_1$, $m_2$ and $m_3$.  We will show that the puzzle can always be solved if and only if
\begin{equation}
\label{eqn:3_sidings}
(h-1)+m_1+m_2+m_3\geq w+M.
\end{equation}
So, for example, with the classic Inglenook track lengths, we have $M=5$ and so~\eqref{eqn:3_sidings} tells us that we need $2+3+3+5\geq w+5$: we can have at most $8$ wagons.

What is the intuition here? 
\begin{itemize}
\item If~\eqref{eqn:3_sidings} is not satisfied, then we can never move the last wagon in the longest siding (or the first wagon in the headshunt if $M=h-1$) to a different siding, because by the time we have removed enough wagons for us to move it, all other sidings are too full. For example,  suppose we have $9$ wagons with classic Inglenook track lengths. To move the last wagon in the longest siding, we can store at most $5$ wagons in the shorter sidings (we need a space for the last wagon to move to) and we can store at most $2$ wagons in the headshunt (we need a space to pick up the wagon). Adding the last wagon in the longest siding makes $8$ wagons in all: there is no good place for that 9$^\text{th}$ wagon. 
\item If~\eqref{eqn:3_sidings} \emph{is} satisfied, then we can move all the wagons. Suppose we have a wagon (say, a banana van) that we want at the end of the longest siding. To solve the shunting problem, we divide the problem into three phases, which roughly go as follows. In Phase~1, we move the banana van to the near-end of a shorter siding. In Phase~2 we move the last wagon from the longest siding and replace it with the banana van. In Phase~3, we move the remaining wagons into the correct places. In Phase~3, we make sure to never move the last wagon in the longest siding. So this phase is essentially a smaller inglenook problem (one less wagon; one siding shortened), which will be an easier puzzle to solve. The proof in the later sections shows all three phases are possible.
\end{itemize}

Once we know that a puzzle can be solved, the next question to ask is: What is the minimum number of shunting moves that are needed to solve a puzzle? For the classic Inglenook Sidings, and starting from a position where the eight wagons all lie in two sidings, there are starting positions that require $17$ moves to solve (if we only care about the ordering of the five wagons chosen). Here, a `move' consists of two movements of the shunting engine: into a siding to pick up or drop off wagons, and then back to the headshunt. See Figure~\ref{fig:inglenook17} for an example with a $17$-move solution, where we are asked to move the wagons so that Wagons $4$ to $8$ lie in the long siding in numerical order and the remaining wagons fill Siding~2 (in any order). In this figure, the wagons in the headshunt are listed from shunting engine to points, and the wagons in each siding are listed from points to buffer stop; each dash indicates a free space. In fact, if we are allowed to start from a position that includes wagons in the headshunt, there are starting positions that need $18$ moves to solve: see Figure~\ref{fig:inglenook18}. If we also require the wagons in Siding 2 to be in numerical order, there are some starting points that require $20$ moves; Figure~\ref{fig:inglenook20} gives an example.

\begin{figure}
{\footnotesize
\[
\begin{array}{c|c|c|cc}
\text{Headshunt}&\text{Siding 1}&\text{Siding 2}&\text{Siding 3}\\\cline{1-4}
\text{ - - - } & \text{ - - - } & \text{ 4 7 8 } & \text{ 1 6 2 3 5  } &\text{(Start)}\\
\text{ 1 6 - } & \text{ - - - } & \text{ 4 7 8 } & \text{ - - 2 3 5  } \\
\text{ 1 6 4 } & \text{ - - - } & \text{ - 7 8 } & \text{ - - 2 3 5  } \\
\text{ - - - } & \text{ 1 6 4 } & \text{ - 7 8 } & \text{ - - 2 3 5  } \\
\text{ 7 8 - } & \text{ 1 6 4 } & \text{ - - - } & \text{ - - 2 3 5  } \\
\text{ 7 8 2 } & \text{ 1 6 4 } & \text{ - - - } & \text{ - - - 3 5  } \\
\text{ 7 8 - } & \text{ 1 6 4 } & \text{ - - 2 } & \text{ - - - 3 5  } \\
\text{ 7 8 3 } & \text{ 1 6 4 } & \text{ - - 2 } & \text{ - - - - 5  } \\
\text{ 7 8 - } & \text{ 1 6 4 } & \text{ - 3 2 } & \text{ - - - - 5  } \\
\text{ 7 8 5 } & \text{ 1 6 4 } & \text{ - 3 2 } & \text{ - - - - -  } \\
\text{ 7 8 - } & \text{ 1 6 4 } & \text{ 5 3 2 } & \text{ - - - - -  } \\
\text{ - - - } & \text{ 1 6 4 } & \text{ 5 3 2 } & \text{ - - - 7 8  } \\
\text{ 1 6 - } & \text{ - - 4 } & \text{ 5 3 2 } & \text{ - - - 7 8  } \\
\text{ 1 - - } & \text{ - - 4 } & \text{ 5 3 2 } & \text{ - - 6 7 8  } \\
\text{ 1 4 - } & \text{ - - - } & \text{ 5 3 2 } & \text{ - - 6 7 8  } \\
\text{ 1 4 5 } & \text{ - - - } & \text{ - 3 2 } & \text{ - - 6 7 8  } \\
\text{ 1 - - } & \text{ - - - } & \text{ - 3 2 } & \text{ 4 5 6 7 8  } \\
\text{ - - - } & \text{ - - - } & \text{ 1 3 2 } & \text{ 4 5 6 7 8  } &\text{(Finish)}
\end{array}
\]}
\caption{A $17$-move Inglenook Sidings solution}
\label{fig:inglenook17}
\end{figure}

\begin{figure}
{\footnotesize
\[
\begin{array}{c|c|c|cc}
\text{Headshunt}&\text{Siding 1}&\text{Siding 2}&\text{Siding 3}\\\cline{1-4}
\text{ 1 6 - } & \text{ - 4 7 } & \text{ 3 5 8 } & \text{ - - - - 2  } &\text{(Start)}\\
\text{ 1 6 4 } & \text{ - - 7 } & \text{ 3 5 8 } & \text{ - - - - 2  } \\
\text{ 1 6 - } & \text{ - - 7 } & \text{ 3 5 8 } & \text{ - - - 4 2  } \\
\text{ 1 6 7 } & \text{ - - - } & \text{ 3 5 8 } & \text{ - - - 4 2  } \\
\text{ 1 - - } & \text{ - - - } & \text{ 3 5 8 } & \text{ - 6 7 4 2  } \\
\text{ 1 3 5 } & \text{ - - - } & \text{ - - 8 } & \text{ - 6 7 4 2  } \\
\text{ - - - } & \text{ 1 3 5 } & \text{ - - 8 } & \text{ - 6 7 4 2  } \\
\text{ 8 - - } & \text{ 1 3 5 } & \text{ - - - } & \text{ - 6 7 4 2  } \\
\text{ 8 1 3 } & \text{ - - 5 } & \text{ - - - } & \text{ - 6 7 4 2  } \\
\text{ - - - } & \text{ - - 5 } & \text{ 8 1 3 } & \text{ - 6 7 4 2  } \\
\text{ 6 7 4 } & \text{ - - 5 } & \text{ 8 1 3 } & \text{ - - - - 2  } \\
\text{ 6 7 - } & \text{ - 4 5 } & \text{ 8 1 3 } & \text{ - - - - 2  } \\
\text{ 6 7 2 } & \text{ - 4 5 } & \text{ 8 1 3 } & \text{ - - - - -  } \\
\text{ 6 7 - } & \text{ 2 4 5 } & \text{ 8 1 3 } & \text{ - - - - -  } \\
\text{ 6 7 8 } & \text{ 2 4 5 } & \text{ - 1 3 } & \text{ - - - - -  } \\
\text{ - - - } & \text{ 2 4 5 } & \text{ - 1 3 } & \text{ - - 6 7 8  } \\
\text{ 2 4 5 } & \text{ - - - } & \text{ - 1 3 } & \text{ - - 6 7 8  } \\
\text{ 2 - - } & \text{ - - - } & \text{ - 1 3 } & \text{ 4 5 6 7 8  } \\
\text{ - - - } & \text{ - - - } & \text{ 2 1 3 } & \text{ 4 5 6 7 8  }&\text{(Finish)}
\end{array}
\]}
\caption{An $18$-move Inglenook Sidings solution}
\label{fig:inglenook18}
\end{figure}

\begin{figure}
{\footnotesize
\[
\begin{array}{c|c|c|cc}
\text{Headshunt}&\text{Siding 1}&\text{Siding 2}&\text{Siding 3}\\\cline{1-4}
\text{ - - - } & \text{ - - - } & \text{ 6 1 8 } & \text{ 5 4 7 2 3  } &\text{(Start)}\\
\text{ 6 1 - } & \text{ - - - } & \text{ - - 8 } & \text{ 5 4 7 2 3  } \\
\text{ - - - } & \text{ - 6 1 } & \text{ - - 8 } & \text{ 5 4 7 2 3  } \\
\text{ 8 - - } & \text{ - 6 1 } & \text{ - - - } & \text{ 5 4 7 2 3  } \\
\text{ 8 5 - } & \text{ - 6 1 } & \text{ - - - } & \text{ - 4 7 2 3  } \\
\text{ 8 5 6 } & \text{ - - 1 } & \text{ - - - } & \text{ - 4 7 2 3  } \\
\text{ - - - } & \text{ - - 1 } & \text{ 8 5 6 } & \text{ - 4 7 2 3  } \\
\text{ 4 7 2 } & \text{ - - 1 } & \text{ 8 5 6 } & \text{ - - - - 3  } \\
\text{ 4 7 - } & \text{ - 2 1 } & \text{ 8 5 6 } & \text{ - - - - 3  } \\
\text{ 4 7 3 } & \text{ - 2 1 } & \text{ 8 5 6 } & \text{ - - - - -  } \\
\text{ 4 7 - } & \text{ 3 2 1 } & \text{ 8 5 6 } & \text{ - - - - -  } \\
\text{ 4 7 8 } & \text{ 3 2 1 } & \text{ - 5 6 } & \text{ - - - - -  } \\
\text{ 4 - - } & \text{ 3 2 1 } & \text{ - 5 6 } & \text{ - - - 7 8  } \\
\text{ 4 5 6 } & \text{ 3 2 1 } & \text{ - - - } & \text{ - - - 7 8  } \\
\text{ - - - } & \text{ 3 2 1 } & \text{ - - - } & \text{ 4 5 6 7 8  } \\
\text{ 3 - - } & \text{ - 2 1 } & \text{ - - - } & \text{ 4 5 6 7 8  } \\
\text{ - - - } & \text{ - 2 1 } & \text{ - - 3 } & \text{ 4 5 6 7 8  } \\
\text{ 2 - - } & \text{ - - 1 } & \text{ - - 3 } & \text{ 4 5 6 7 8  } \\
\text{ - - - } & \text{ - - 1 } & \text{ - 2 3 } & \text{ 4 5 6 7 8  } \\
\text{ 1 - - } & \text{ - - - } & \text{ - 2 3 } & \text{ 4 5 6 7 8  } \\
\text{ - - - } & \text{ - - - } & \text{ 1 2 3 } & \text{ 4 5 6 7 8  }&\text{(Finish)}
\end{array}
\]}
\caption{A $20$-move Inglenook Sidings solution, with all wagons ordered}
\label{fig:inglenook20}
\end{figure}

\subsection{Related literature}

The academic study of the rearrangement of rail wagons goes back at least as far as 1968: in his \emph{Art of Computer Programming}, Donald Knuth~\cite{Knuth68} describes Dijkstra's idea of visualising stack-sorting in terms of reordering wagons. Knuth's characterisation of those permutations that can be sorted using a single stack has been highly influential, in particular leading to the theory of permutation patterns (see, for example, Kitaev~\cite{Kitaev}) and to the literature on sorting using queues and stacks (see, for example, B\'ona~\cite{Bona}).

There is a significant literature on the sorting of wagons; see~\cite{DiStefanoMaue,GattoMaue,HansmannZimmermann} for surveys. In contrast to this paper, the literature often formalises sidings as stacks (or replaces sidings by loops that can be accessed from both ends, modelled as double ended queues), and so places no limit to the number of wagons that can be moved into a siding. (A paper by Atkinson, Livesey and Tulley~\cite{AtkinsonLivesey} is one significant exception, that considers (in particular) sorting using bounded stacks.) Also, the most common definition of a move is different to that considered in this paper: a move consists of taking the wagons in a fixed siding, drawing them into the headshunt, and then distributing them one at a time (in order, from points to shunting engine) into the sidings. This definition of a move is motivated by \emph{hump yards}, where there is a low artificial hill, or hump, between the shunting engine and the points. Arriving wagons are pushed just beyond this hill, and then individually roll under gravity into their chosen sidings. To sort wagons, the wagons in a siding are drawn back over the hump, and this process is repeated. Jacob, M\'arton, Maue and Nunkesser~\cite{JacobMarton} consider the algorithmic complexity of sorting with this formalisation of a hump yard, proving hardness and approximability results for the problem of finding the optimal number of moves required; they consider the case when the length of a siding is bounded, as well as the case where the number of sidings is limited. Dhalhaus, Horak, Miller and Ryan~\cite{DahlhausHorak} provide results on the number of sidings needed to sort a train in a related model, answering a question of Knuth. Finally, Di Stefano and Ko\v ci~\cite{DiStefanoKoci} consider a model motivated by the problem of filling a tram depot at night so that trams may leave the following morning with minimal shunting. Here, the sidings are bounded in number and length, but the moves are different in nature. 

\subsection{Structure of the paper}

The remainder of this paper is structured as follows. In Section~\ref{sec:statements} we provide more precise definitions for the puzzles we consider, and we state our results precisely. We look at a more general situation than that described in Subsection~\ref{sub:nontechnical} above, by not restricting ourselves to the case of three sidings. We also define a simpler (and much duller) puzzle played with piles of cards, and show that the inglenook puzzle can always be solved when the same is true for this card puzzle. In Section~\ref{sec:card}, we establish exactly when this card puzzle can always be solved (Theorem~\ref{thm:card}), and this will give the result (Theorem~\ref{thm:inglenook}) on the solubility of inglenook puzzles we are aiming for. Finally, in Section~\ref{sec:numberofmoves} we provide a result (Theorem~\ref{thm:inglenook_moves}) on the number of moves needed to solve an inglenook puzzle.

\section{Problem statement and a simpler puzzle}
\label{sec:statements}

Let $w$, $s$, $h$, $m_1,m_2,\ldots,m_s$ be positive integers with $w< h+m_1+m_2+\cdots+m_s$. We formally define an inglenook puzzle  as follows. We have $s$ sidings of lengths $m_1,m_2,\ldots,m_s$, with headshunt length $h$ and $w$ wagons. Let $W$ be the set of wagons (so $|W|=w)$.
\begin{definition}
A \emph{position} $p$ is a sequence $p=(W_0,W_1,\ldots,W_s)$ of ordered subsets of wagons such that 
\begin{enumerate}
\item $\bigcup_{i=0}^s W_i=W$ as sets,
\item $W_i$ and $W_j$ are disjoint whenever $i\not=j$,
\item $|W_i|\leq m_i$ for $1\leq i\leq s$, and
\item $|W_0|\leq h$.
\end{enumerate}
\end{definition}
We think of $W_0$ as the subset of wagons in the headshunt, ordered \emph{from shunting engine to points}; and for $i>1$, $W_i$ as the subset of wagons in Siding~$i$, ordered \emph{from points to buffer stop}. The above conditions just say: all wagons are in a siding or the headshunt; no wagon is in two sidings, or both a siding and the headshunt; there are no more than $m_i$ wagons in Siding~$i$; there are no more than $h$ wagons in the headshunt. In Figure~\ref{fig:inglenook17}, the position after one move is $p=((1,6),\varepsilon,(4,7,8),(2,3,5))$, where $\varepsilon$ is the empty ordered set.

\begin{definition}
A \emph{move} is a pair $(p_1,p_2)$ of distinct positions
\[
p_1=(W_0,W_1,\ldots,W_s)\text{ and }p_2=(W'_0,W'_1,\ldots,W'_s)
\]
such that for some $r\in\{1,2,\ldots,s\}$:
\begin{enumerate}
\item $W_i=W'_i$ for $1\leq i\leq s$ with $i\not=r$, and
\item the concatenation of $W_0$ and $W_r$ is equal to the concatenation of $W'_0$ and $W'_r$.
\end{enumerate}
\end{definition}
So a move transfers $||W_0|-|W'_0||$ wagons between the headshunt and Siding~$r$. If $|W_0|-|W'_0|>0$ then wagons are transferred to the headshunt, otherwise the wagons are transferred to the siding.


\begin{definition}
The \emph{inglenook graph} $\Gamma(w,s,h,m_1,m_2,\ldots,m_s)$ is a graph with vertex set $V$ and edge set $E$, where $V$ is the set of positions and $E$ is the set of moves defined above.
\end{definition}

When $(p_1,p_2)$ is a move, so is $(p_2,p_1)$. So $\Gamma(w,s,h,m_1,m_2,\ldots,m_s)$ is an undirected graph. 

An inglenook puzzle is played on the inglenook graph $(V,E)$. A set $S\subset V$ of \emph{starting positions} and a set $F\subseteq V$ of \emph{finishing positions} is specified. We begin at a vertex $p\in S$, and solve the puzzle by finding a path in the graph to any vertex in $E$. The puzzle can \emph{always be solved} if and only if every $p\in S$ is contains a finishing position in its connected component (so from every starting position, there is a path in the inglenook graph to a finishing position. The classic Inglenook Sidings puzzle takes $S$ to be the set of positions with the headshunt and Siding~1 empty, with Sidings~2 and~3 containing $3$ and $5$ wagons respectively; the subset $F$ is chosen to be the subset of $S$ where (for example) wagons $4$ to $8$ appear in order in Siding~3. The common variant of the inglenook puzzle also mandates an order for the wagons in Siding~2 (wagons~1 to~3 appearing in order, for example): $F$ contains just one position in this case. We write $I(w,s,h,m_1,m_2,\ldots,m_s;S,F)$ for this puzzle.

We make two assumptions on the sets $S$ and $F$. First, we assume that if $p\in S$ is a starting position, then the position $p'$ obtained from $p$ by permuting the wagons amongst themselves in an arbitrary way is also a starting position (so $S$ is closed under permuting wagons). Second, we assume that for some $z\in\{1,2,\ldots,s\}$, all the finishing positions $F$ have the same wagon, say wagon $\omega\in W$, as the last wagon (nearest the buffer stop) in Siding~$z$. This condition holds, for example, when the ordered subset of wagons in a particular siding are mandated. The assumptions for $S$ and $F$ hold for both the classic Inglenook Sidings puzzle and the variant when the final position of all wagons is specified; if these assumptions hold, we say that $F$ and $S$ are \emph{natural} and that $I(w,s,h,m_1,m_2,\ldots,m_s;S,F)$ is a \emph{natural} inglenook puzzle. The following theorem characterises when natural inglenook puzzles can always be solved:

\begin{theorem}
\label{thm:inglenook}
Let $w$, $s$, $h$, $m_1,m_2,\ldots,m_s$ be positive integers with $w< h+m_1+m_2+\cdots+m_s$. Let $S$ and $F$ be natural starting and finishing sets (see above). The natural inglenook puzzle $I(w,s,h,m_1,m_2,\ldots,m_s;S,F)$ (with $s$ sidings of lengths $m_1,m_2,\ldots,m_s$, with headshunt length $h$, with $w$ wagons, starting positions $S$ and finishing positions $F$) can always be solved when the inequalities $w>1$, $s>1$, $h>1$ and
\begin{equation}
\label{eqn:main_condition}
h-1+m_1+m_2+\cdots+m_s\geq w+\max\{h-1,m_1,m_2,\ldots,m_s\}
\end{equation}
all hold. The inglenook puzzle can always be solved when the conditions $w>1$, $s>2$, $h=1$ and~\eqref{eqn:main_condition} all hold. The inglenook puzzle can also always be solved when $w=1$. In all other cases, the puzzle cannot always be solved.
\end{theorem}

We note that most of these conditions only come into play in extreme cases (fewer than three sidings; a very small headshunt; a very small number of wagons). Away from these cases, the inequality~\eqref{eqn:main_condition} is the only condition that needs to be satisfied.
We prove this theorem by analysing a simpler Cards in Piles puzzle (simpler because only one object is moved at once, and because the headshunt and sidings are not distinguished). The puzzle is played using a set $W$ of $w$ cards labelled with different symbols (maybe a picture of a wagon!), and we define it as follows. Let $w$, $s$ and $m_0,m_1,\ldots,m_s$ be positive integers such that $w\leq \sum_{i=0}^sm_i$.

\begin{definition}
A position, or \emph{state}, of the puzzle is a distribution of the $w$ cards into $s+1$ piles, with (for $i\in\{0,1,2,\ldots,s\}$) Pile~$i$  containing between $0$ and $m_i$ cards. States with cards in a different order in a pile are thought of as different. (More formally, a state is a sequence $(P_0,P_1,\ldots,P_s)$ of ordered subsets of $W$, where the subsets $P_i$ partition $W$ and $|P_i|\leq m_i$ for all $i$.)
\end{definition}
\begin{definition}
A \emph{move} in the puzzle consists of picking one card from the top of a pile, and placing it on the top of another pile. It is forbidden for Pile~$i$ to ever contain more than $m_i$ cards, so each move results in another state. (Formally, a move is a pair $(p_1,p_2)$ of states, where $p_1$ is the state before the card is moved and $p_2$ is the state that results after the card is moved.)
\end{definition}
\begin{definition}
The \emph{Cards in Piles graph} $G(w,s,m_0,\ldots,m_s)$ is the (undirected) graph with vertex set $V$ equal to the set of states, and edge set $E$ equal to the set of moves. 
\end{definition}
Define the set $S$ of starting states by $S=V$ (so we start in any state).  Define the set $F\subseteq V$ of finishing states to be any subset of size $1$. The puzzle begins with the cards in a starting state. The aim of the puzzle is to reach the finishing state. We write $C(w,s,m_0,\ldots,m_s;F)$ for this puzzle. We see that $C(w,s,m_0,\ldots,m_s;F)$ can always be solved if and only if the Cards in Piles graph $G(w,s,m_0,\ldots,m_s)$ is connected. There is a simple condition that tells us when this happens:

\begin{lemma}
\label{lem:connected}
The Cards in Piles graph $G(w,s,m_0,\ldots,m_s)$ is connected when $w=1$. When $w>1$ and $s=1$ the graph is disconnected. When $w>1$ and $s>1$, the graph is connected if and only if
\[
\sum_{i=0}^s m_i\geq w+\max\{m_i:0\leq i\leq s\}.
\]
\end{lemma}

We will prove Lemma~\ref{lem:connected} in the next section. The following theorem immediately follows from Lemma~\ref{lem:connected} and the discussion above:

\begin{theorem}
\label{thm:card}
The Cards in Piles puzzle $C(w,s,m_0,\ldots,m_s;F)$ can always be solved when $w=1$. When $w>1$ the puzzle can always be solved except in the following two situations:
\begin{enumerate}
\item when $s=1$;
\item when $\sum_{i=0}^s m_i<w+\max\{m_i:0\leq i\leq s\}$.
\end{enumerate}
\end{theorem}

But we are more interested in inglenook puzzles: we now show that Theorem~\ref{thm:inglenook} also follows from Lemma~\ref{lem:connected}.

\begin{proof}[Proof of Theorem~\ref{thm:inglenook}]
It is easy to see that $I(w,s,h,m_1,m_2,\ldots,m_s;S,F)$ can always be solved when there is only one wagon! So the theorem holds when $w=1$.

Now assume that $w>1$ and $s=1$. So we have only one siding. Since our inglenook puzzle is natural, all finishing positions have some wagon $\omega\in W$ as the last wagon in that siding. We may list the wagons by first listing those in the headshunt from shunting engine to the points, and then listing those in the siding from points to buffer stop. Shunting cannot change the order of the wagons in this list. In particular, the last wagon in the siding cannot be swapped for another. Since $w>1$ and our inglenook puzzle is natural, there is a starting position where the last wagon is not $\omega$, so the puzzle cannot always be solved. So the theorem holds when $s=1$.

Now assume that $w>1$, $s=2$ and $h=1$. We may list the wagons by firstly listing those in the first siding from buffer stop to points, then the wagon (if there is one) in the headshunt, and then listing those in the second siding from points to buffer stop. As before, shunting cannot change the order of the wagons in this list, so the puzzle cannot always be solved. So the theorem holds when $h=1$ and $s=2$.

We now assume that $w>1$ and $s>1$ (with no assumption on $h$). 
Suppose that
\begin{equation}
\label{eqn:infeasible}
h-1+\sum_{i=1}^sm_i<w+\max\{h-1,m_1,m_2,\ldots,m_s\}.
\end{equation}
Since our puzzle is natural, all finishing positions have the same wagon $\omega$ at the end of some siding, say Siding~$z$.

First, suppose that $\max\{h-1,m_1,m_2,\ldots,m_s\}=h-1$. Then~\eqref{eqn:infeasible} implies that $\sum_{i=1}^sm_i<w$. But in this case there are not enough spaces in the sidings for all of the wagons; in particular, all positions have a wagon next to the shunting engine, and this wagon can never be changed. So the puzzle cannot always be solved in this case, because there are starting positions with $\omega$ adjacent to the shunting engine, and this wagon cannot be moved into Siding~$z$. Second, suppose that $\max\{h-1,m_1,m_2,\ldots,m_s\}=m_j$ for some $j>0$. Then~\eqref{eqn:infeasible} implies that $h-1+\sum_{i\in\{1,2,\ldots,s\}\setminus \{j\}}m_i<w$, and so $h+\sum_{i\in\{1,2,\ldots,s\}\setminus \{j\}}m_i\leq w$. When this inequality is strict, so $h+\sum_{i\in\{1,2,\ldots,s\}\setminus \{j\}}m_i<w$, the wagon next to the buffer stop in Siding~$j$ cannot be moved, since any position after this wagon is transferred to the headshunt must have no wagons in Siding~$j$, but no such position can exist as there is not enough space in the other sidings and the headshunt to accommodate all wagons. Similarly, when $h+\sum_{i\in\{1,2,\ldots,s\}\setminus \{j\}}m_i=w$, this wagon cannot be transferred to another siding, because whenever it is moved into the headshunt all other sidings and the headshunt are full: the only option is to place the wagon back next to the buffer stop in Siding~$j$ as the next move. If $j=z$, the starting positions that do not have $\omega$ as the last wagon in Siding $j$ cannot be solved. If  $j\not=z$, the starting positions that have $\omega$ as the last wagon in Siding~$j$ cannot be solved. So in either case, the inglenook puzzle cannot always be solved.

To establish the theorem, we show that the puzzle can always be solved in the remaining cases. We may now assume that $w>1$, $s>1$ and $h-1+\sum_{i=1}^sm_i\geq w+\max\{h-1,m_1,m_2,\ldots,m_s\}$. If $h=1$, we assume that $s>2$. It is sufficient to show that the inglenook graph $\Gamma(w,s,h,m_1,m_2,\ldots,m_s)$ is connected.

First, assume that $h>1$.

We say that an inglenook position $\bx$ is \emph{convertible} if the headshunt is not full. So $\bx$ is convertible if and only if the number of wagons in the headshunt is at most $h-1$. We note that if $\bx$ is not convertible, then there is a move to a convertible position (indeed, all moves are to convertible positions, as one or more wagons must be moved from the headshunt to a siding with space before anything else can be done). In particular, in order to prove that $\Gamma(w,s,h,m_1,m_2,\ldots,m_s)$ is connected it is sufficient to show that all convertible positions lie in the same connected component.

Define $m_0=h-1$. Since we are assuming that $h>1$, we see that $m_0$ is positive. There is a one-to-one correspondence between convertible positions of the inglenook graph $\Gamma(w,s,h,m_1,m_2,\ldots,m_s)$ and states of the Cards in Piles graph $G(w,s,m_0,\ldots,m_s)$: the $w$ wagons correspond to the $w$ cards; the wagons in Siding~$i$ correspond to Pile~$i$ in the card puzzle with the wagon closest to the buffer stop being on the bottom of the pile; wagons in the headshunt correspond to Pile~0, with the wagon \emph{closest to the shunting engine} being on the bottom of the pile.

Lemma~\ref{lem:connected} shows that $G(w,s,m_0,\ldots,m_s)$ is connected, so there is a path from any state to any other. By definition, an edge in $G(w,s,m_0,\ldots,m_s)$ corresponds to a move in the Cards in Piles puzzle. Note that moving a card to or from Pile~0 corresponds to picking up or dropping off a single wagon from the corresponding siding in the associated (convertible) position in $\Gamma(w,s,h,m_1,m_2,\ldots,m_s)$. Also note that a move between any other pair of piles (Piles~$i$ and~$j$, say) corresponds to two moves in the inglenook puzzle: picking up a wagon from Siding~$i$ and then dropping it in Siding~$j$. Note that a convertible position never has a full headshunt, so these are valid moves in the inglenook puzzle. Therefore an edge in $G(w,s,m_0,\ldots,m_s)$ corresponds to a path of length one or two in $\Gamma(w,s,h,m_1,m_2,\ldots,m_s)$, with the corresponding convertible positions as end points. Since $G(w,s,m_0,\ldots,m_s)$ is connected, all the convertible positions in $\Gamma(w,s,h,m_1,m_2,\ldots,m_s)$ lie in the same connected component as required. Since all non-convertible positions are adjacent to convertible positions the inglenook graph is connected and so the theorem holds when $h>1$.

Finally, assume that $h=1$ and $s>2$. We proceed as we did when $h>1$, but observe that the convertible positions in the inglenook graph correspond to states in the Cards in Piles graph $G(w,s-1,m_1,\ldots,m_s)$ (one fewer pile, since the headshunt is always empty in a convertible position). By Lemma~\ref{lem:connected}, this graph is connected and so the inglenook puzzle can always be solved in this situation. So the theorem follows from Lemma~\ref{lem:connected}, as claimed.
\end{proof}

\section{The Cards in Piles graph}
\label{sec:card}

The aim of this section is to establish when the Cards in Piles graph is connected.

\begin{proof}[Proof of Lemma~\ref{lem:connected}]
It is easy to see that $G(w,s,m_0,\ldots,m_s)$ is always connected when $w=1$. So for the rest of this proof, we assume that $w\geq 2$.

We first check that the graph is disconnected in the cases listed in the statement of the lemma.

Suppose that $s=1$. If we read Pile~0 from bottom to top, and then Pile~1 from top to bottom, we get an ordering of the cards; any move does not change this order. Since $w>1$, there are states with different orderings, and so the graph $G(w,s,m_0,\ldots,m_s)$ is disconnected.

Suppose that $\sum_{i=0}^s m_i<w+\max\{m_i:0\leq i\leq s\}$. Let $j\in\{0,1,\ldots,s\}$ be chosen so that $m_j$ is as large as possible, so $\sum_{i=0}^s m_i<w+m_j$. This last inequality can be rewritten as
\[
\sum_{i\in\{0,1,\ldots,s\}\setminus\{j\}}m_i\leq w-1.
\]
But in this situation, we can never move the card at the bottom of Pile~$j$. (Suppose were able to reach a state where it could be moved. After this move, all $w$ cards must be distributed amongst the other piles, and so $w\leq \sum_{i\in\{0,1,\ldots,s\}\setminus\{j\}}m_i$. This contradicts the inequality above.) Since $w>1$, there are states with a different card at the bottom of Pile~$j$, and so the graph is disconnected.

Assume that $w>1$, $s>1$ and $\sum_{i=0}^s m_i\geq w+\max\{m_i:0\leq i\leq s\}$. To establish the lemma, we need to show that the graph is connected.

Suppose we are in a state where $w_i$ cards are in Pile~$i$. Then at most $m_i-w_i$ extra cards can be added to the pile: we say that Pile~$i$ has $m_i-w_i$ \emph{spaces} available. Note that the space of a given pile depends on the state, but the \emph{total space} across all piles does not. Indeed, in our situation
\begin{equation}
\label{eqn:space_eqn}
\sum_{i=0}^s(m_i-w_i)=\left(\sum_{i=0}^s m_i\right)-w\\
\geq \max\{m_i:0\leq i\leq s\}.
\end{equation}
In particular, the total space is positive, and if the total space is $1$ then $m_i=1$ for all $i$.

We argue that the graph is connected when the total space is $1$. In this situation, a state consists of $s$ piles with one card, and one pile with $0$ cards. A single move changes the position of a zero pile; three appropriately chosen moves, involving the zero pile, swaps any pair of cards we wish whilst keeping the rest of the cards unchanged. Using these moves (and the fact that any permutation can be written as a composition of transpositions) we see that the graph is connected.

Now suppose that the total space is at least $2$. We prove that the graph is connected by induction on the number $w$ of cards. When $w=1$, this is obvious. So we assume that $w>1$ and that the lemma holds for all smaller values of $w$. If $m_i=1$ for all $i$, the graph is connected by a similar argument to the total space $1$ case above. So we may assume that $m_k\geq 2$ for some $k\in\{0,1,\ldots,s\}$.

Let $\bx$ and $\by$ be two states. We need to show that there is a path in $G(w,s,m_0,\ldots,m_s)$ from $\bx$ to $\by$.

There is a state $\by'$ in the neighbourhood of $\by$ such that Pile~$k$ is not empty (starting at $\by$, move any card into Pile~$k$ from another pile if necessary). Let $b\in W$ be the card at the bottom of Pile~$k$ when in position~$\by'$.

It is sufficient to find a path in $G(w,s,m_0,\ldots,m_s)$ from $\bx$ to $\by'$. We claim that there is a path from $\bx$ to a state $\bz$ such that the card at the bottom of Pile~$k$ is equal to $b$ (and so agrees with the card at the bottom of Pile~$k$ in $\by'$). We will see how this is possible below. We then find a path from $\bz$ to $\by'$ by restricting our moves: we never move the bottom card in Pile~$k$. This is equivalent to playing an instance of the puzzle $C(w-1,s,m'_0,m'_1,\ldots,m'_s)$, where
\[
m'_i=\begin{cases} m_i-1&\text{ when }i=k\\
m_i&\text{ otherwise.}
\end{cases}
\]
(Note that we are using the fact that $m_k\geq 2$ here: this condition implies that $m'_k$ is positive.) By the inductive hypothesis, $G(w-1,s,m'_0,m'_1,\ldots,m'_s)$ is connected and so a path from $\bz$ to $\by'$ exists. So our lemma holds once our claim is proved.

We now establish our claim. If the bottom card in Pile~$k$ in $\bx$ is $b$, there is nothing to do. So we assume that $b$ does not lie at the bottom of Pile~$k$.

First, by starting at position $\bx$, we may find a path to a state $\bx'$ where $b$ lies in Pile~$k$. We see this as follows. Let $b$ lie in Pile~$\ell$. If $\ell=k$, there is nothing to do, so suppose that $\ell\not=k$. If Pile~$k$ has no space, we move a card from Pile~$k$ to some other pile. We then transfer the cards above $b$ from Pile~$\ell$ to other piles, filling spaces in Pile~$k$ last (if at all). Writing $w_i$ for the number of cards in Pile~$i$ at the start of this process, there are at most $w_\ell-1$ cards above $b$ in Pile~$\ell$, and the space available in other piles is, using~\eqref{eqn:space_eqn}, 
\begin{align}
\nonumber
\sum_{i\not=\ell} (m_i-w_i)&=\left(\sum_{i=0} ^sm_i\right)-w-(m_\ell-w_\ell)\\
\nonumber
&\geq \max\{m_i:0\leq i\leq s\}-m_\ell+w_\ell\\
\label{eqn:space}
&\geq w_\ell\\
&>w_{\ell}-1\nonumber
\end{align}
so this is possible, and there is still least one space in Pile~$k$ at the end of this process. Since $b$ now lies at the top of a pile, we may transfer $b$ to Pile~$k$ to form the state $\bx'$, as required. If $b$ is now the bottom card in Pile~$k$, we may take $\bz=\bx'$ and there is nothing more to do. So we may assume that $b$ is not on the bottom of Pile~$k$.

Second, starting at $\bx'$ and transferring a card to Pile~$k$ if necessary, we produce a state such that the total space in the remaining piles (in other words, those not equal to Pile~$k$) is at least $2$. If necessary, transferring a single card between these remaining piles produces a state $\bx''$ that has the additional property that two of these remaining piles (Pile~$u$ and Pile~$v$, say) have non-zero space. (We are using the fact that $s\geq 2$ here.)

Third, starting at state $\bx''$ we transfer all cards above $b$ from Pile~$k$ to other piles, in such a way that Pile~$u$ and Pile~$v$ still both have space. This is possible since, writing $w_i$ for the number of cards in Pile~$i$ in $\bx''$, there are at most $w_k-2$ cards above $b$ (as $b$ is not the bottom card in its pile) and since $\sum_{i\not=k} (m_i-w_i)\geq w_k$ by~\eqref{eqn:space}. We write $\bx'''$ for the state that results.

Fourth, starting at state $\bx'''$ we transfer $b$ to Pile~$u$, and then transfer the remaining cards from Pile~$k$. This is done so that $b$ remains on the top of a pile after this process is complete: transferring $b$ from Pile~$u$ to (the top of) Pile~$v$ if $v$ has one space remaining at any point ensures this property. Finally, once Pile~$k$ is empty we transfer $b$ to the bottom of Pile~$k$. This is the state $\bz$, as required, and our claim is established. The lemma then follows.
\end{proof}

\section{How many moves are needed?}
\label{sec:numberofmoves}

This section contains results about the maximum number of moves we need, provided the puzzle can always be solved. We begin by considering the Cards in Piles puzzle, and then apply these results to inglenook puzzles.

For a connected graph $G$, recall that the \emph{diameter} of $G$ is the largest distance between any pair of vertices. Write $f(w)$ for the largest diameter of the graph $G(w,s,m_0,\ldots,m_s)$,  where $s,m_0,\ldots,m_s$ are arbitrary subject to the graph being connected.

\begin{lemma}
\label{lem:connected_upper}
Define $f(w)$ as above. Then
\[
f(w)\leq w^2+6w-6.
\]
\end{lemma}

\begin{proof}
The proof of Lemma~\ref{lem:connected} shows that for $w>1$
\[
f(w)\leq \begin{cases}
3(w-1)+1&\text{ if the total space is $1$,}\\
2w+5+f(w-1)&\text{ otherwise.}
\end{cases}
\]
Since $f(1)=1$, we find that $f(w)\leq w^2+6w-6$ using induction on $w$. So the lemma follows.
\end{proof}

\begin{lemma}
\label{lem:connected_lower}
The diameter of $G(w,2,w-1,w-1,1)$ is at least $(w^2-1)/4$.
\end{lemma}

\begin{proof}
Let $X$ be the set of all sequences $\underline{x}=(x_1,x_2,\ldots,x_w)$ over $W$ of length~$w$ with all entries distinct, so $X$ may be thought of as the set of all orderings of the elements of $W$. Let $\underline{x},\underline{y}\in X$ be sequences with $\underline{x}=(x_1,x_2,\ldots,x_w)$ and $\underline{y}=(y_1,y_2,\ldots,y_w)$. For a card $\omega\in W$, there exist unique $i,j\in\{1,2,\ldots,w\}$ such that $x_i=y_j=\omega$. We write $d_\omega(\underline{x},\underline{y})=|i-j|$, and we define the \emph{distance} from $\underline{x}$ to $\underline{y}$ to be
\[
d(\underline{x},\underline{y})=\sum_{\omega\in W} d_\omega(\underline{x},\underline{y}).
\]
It is not hard to see that this notion of distance is a metric on $X$.

Consider the subset $H$ of states in $G(w,2,w-1,w-1,1)$ with Pile~$2$ empty. Define a map $\phi:H\rightarrow X$ by, for $\bz\in H$, setting $\phi(\bz)$ to be the ordering of $W$ determined by the cards in Piles~$0$ and~$1$, ordering Pile~$0$ from bottom to top, and then Pile~$1$ from top to bottom.

Choose a state $\bz\in H$, and write $\phi(\bz)=(x_1,x_2,\ldots,x_w)$. Let $\bz'\in H$ be chosen so that $\phi(\bz')=(x_w,x_{w-1},\ldots,x_1)$ (so the ordering of cards is reversed compared to  $\bz\in H$). We have that 
\begin{align*}
d(\phi(\bz),\phi(\bz'))&=\sum_{i=1}^w |w+1-2i|\\
&= \begin{cases}
w^2/2&\text{ when $w$ is even}\\
(w^2-1)/2&\text{ when $w$ is odd}
\end{cases}\\
&\geq (w^2-1)/2.
\end{align*}

Fix a shortest path from $\bz$ to $\bz'$, and let $\bz_0,\bz_1,\ldots,\bz_\ell$ be the elements on this path that lie in $H$, in order. So $\bz_0=\bz$ and $\bz_\ell=\bz'$. Consider the segment of the path between $\bz_j$ and $\bz_{j+1}$ for some fixed integer $j$. Let $\underline{x}=\phi(\bz_j)$ and $\underline{y}=\phi(\bz_{j+1})$.

We claim that the length of the segment of the path from $\bz_j$ to $\bz_{j+1}$ is at least $d(\underline{x},\underline{y})/2$.  This is sufficient to prove the lemma, since then the length of the path is at least
\[
\tfrac{1}{2}\sum_{j=0}^{\ell-1}d(\phi(\bz_j),\phi(\bz_{j+1}))\geq \tfrac{1}{2}d(\phi(\bz_0),\phi(\bz_{\ell}))=\tfrac{1}{2}d(\phi(\bz),\phi(\bz'))\geq (w^2-1)/4.
\]
Since the path was as short as possible, the diameter of the graph is at least $(w^2-1)/4$, as required.

We now aim to prove our claim. 
If $\bz_j$ and $\bz_{j+1}$ are consecutive elements in the path, the move from $\bz_j$ to $\bz_{j+1}$ transfers cards between Piles~$0$ and~$1$ and so $\underline{x}=\underline{y}$. In particular, $d(\underline{x},\underline{y})=0$ and so the claim follows trivially in this case. Now suppose that $\bz_j$ and $\bz_{j+1}$ are not consecutive. The segment of the path from $\bz_j$ to $\bz_{j+1}$ begins and ends with a transfer to and from Pile~$2$, but the remaining moves transfer cards between Piles~$0$ and $1$, and so do not affect the order of the cards in those piles. The net result of this process is that one card is removed from the sequence $\underline{x}$ and placed in a different position to form $\underline{y}$. So there exist $u,v\in\{1,2,\ldots,w\}$ such that $u<v$ and either
\[
y_i=\begin{cases} x_i&\text{ when }i<u\text{ or }i>v\\
x_{i+1}&\text{ when }u\leq i<v\\
x_u&\text{ when }i=v
\end{cases}
\]
or
\[
y_i=\begin{cases} x_i&\text{ when }i<u\text{ or }i>v\\
x_v&\text{ when }i=u\\
x_{i-1}&\text{ when }u\leq i<v.
\end{cases}
\]
In either case, we see that $d(\underline{x},\underline{y})=2(v-u)$, since one card moves $v-u$ positions in the ordering, and $v-u$ cards move $1$ position in the ordering. But there are at least $(v-u)$ edges in the segment of the path between $\bz_j$ and $\bz_{j+1}$ in $G(w,2,w-1,w-1,1)$. To see this, first suppose $x_u$ is the card that is moved to Pile~$2$ as the first move in the segment. Then $x_u$ is at the top of either Pile~$0$ or Pile~$1$ in $\bz_j$ (otherwise it cannot be moved) and so there are $u-1$ cards in Pile~$0$ after the first move. Similarly, there must be $v-1$ cards in Pile~$0$ before the last move in the segment. So there must be at least $(v-1)-(u-1)=v-u=\frac{1}{2}d(\underline{x},\underline{y})$ moves in the segment that transfer a card between Piles~$0$ and~$1$, and our claim follows in this case. Similarly if $x_v$ is the card that is moved at the start of the segment between $\bz_j$ and $\bz_{j+1}$, there are $v-1$ cards in Pile~$0$ after this move, and $u-1$ cards in Pile~$0$ before the last move in the segment. So there are at least $(v-1)-(u-1)=v-u=\frac{1}{2}d(\underline{x},\underline{y})$ moves in the segment in this case also, and our claim follows.
\end{proof}

The following corollary follows immediately from Lemmas~\ref{lem:connected_upper} and~\ref{lem:connected_lower}:

\begin{corollary}
\label{cor:cards_in_piles_moves}
When it can be always be solved, the Cards in Piles puzzle with~$w$ cards can be solved using at most $w^2+6w-6$ moves. There are puzzles using $w$ cards that require at least $(w^2-1)/4$ moves to solve.
\end{corollary}

\begin{theorem}
\label{thm:inglenook_moves}
A natural inglenook puzzle with $w$ wagons can be solved in at most $2w^2+12w-10$ moves, provided the puzzle can always be solved. There are starting positions for natural inglenook puzzles that require at least $(w^2-1)/2$ moves to solve.
\end{theorem}
\begin{proof}
 Each position in an inglenook graph has a convertible position in its neighbourhood, and an edge in the associated Cards in Piles graph corresponds to at most $2$ moves in the inglenook graph. So  the diameter of the inglenook graph is at most $2f(w)+2$. Thus (provided the puzzle is always solvable) a natural inglenook puzzle can be solved in at most $2w^2+12w-10$ moves, by Lemma~\ref{lem:connected_upper}.
 
Lemma~\ref{lem:connected_lower} shows that there are states $\bz,\bz'$ in the Cards in Piles graph $G(w,2,w-1,w-1,1)$ that have no paths of length less than $(w^2-1)/4$ between them. Consider an inglenook puzzle with three sidings of lengths $w-1$, $w-1$ and $1$ respectively, and a headshunt of length $1$. Let $\bx$ and $\bx'$ be the convertible positions in the inglenook graph $\Gamma(w,3,1,w-1,w-1,1)$ that correspond to $\bz,\bz'\in G(w,2,w-1,w-1,1)$. Let $F=\{\bx'\}$ and $S=\Gamma(w,3,1,w-1,w-1,1)$. It is sufficient to show that the inglenook puzzle $I(w,3,1,w-1,w-1;S,F)$ requires at least $(3w^2+w)/2$ moves to solve, when we are given the starting position $\bx\in S$. To see this, note that (because the headshunt can contain only one wagon) there is a one-to-one correspondence between edges in the Cards in Piles graph and pairs of consecutive moves starting (and finishing) at convertible positions in the inglenook graph. Thus paths between convertible positions in $\Gamma(w,3,1,w-1,w-1,1)$ are twice the length of the corresponding paths in $G(w,2,w-1,w-1,1)$. If we could solve the inglenook puzzle in fewer than $(w^2-1)/2$ moves, there is a path of length less than  $(w^2-1)/2$ between positions $\bx$ and $\bx'$ in $\Gamma(w,3,1,w-1,w-1,1)$: the corresponding path between states $\bz$ and $\bz'$ in $G(w,2,w-1,w-1,1)$ would have length less than $(w^2-1)/4$, contradicting our choice of $\bz,\bz'$. So we cannot solve the inglenook puzzle in fewer than $(w^2-1)/2$ moves, and the theorem follows.
\end{proof}

It would be very interesting to give more precise results on the maximum number of moves needed for an inglenook puzzle. For the classic Inglenook Sidings puzzle, the above argument shows that at most $214$ moves are needed, but we can actually get by with just $20$ moves (even if we require all wagons to be in a fixed order). So there is scope for improvement. However, the argument above does show that this strategy of `working from the back wagons forwards' is reasonably efficient, even though there is no reason for it to be optimal. Are there better general methods for solving a natural inglenook puzzle? In particular, what is the smallest value $c$ such that we can always solve the puzzle in $cw^2+o(w^2)$ moves?

\end{document}